\documentclass[11pt,twoside]{article}
\usepackage{graphicx}
\usepackage{rotating}
\usepackage{amsmath}
\usepackage{amsfonts}
\usepackage{amssymb}
\usepackage{amsthm}
\usepackage{tabularx}
\usepackage[all]{xy}

\newtheorem{thm}{Theorem}[section]

 \newtheorem{cor}[thm]{Corollary}
 \newtheorem{exa}[thm]{Example}
 \newtheorem{rem}[thm]{Remark}
 
  \newtheorem{defn}[thm]{Definition}

\begin{document}

 \centerline{\large{\textbf{FIBRATIONS PROPERTY OF EMBEDDING  MAPS     }}}
 \centerline{\large{\textbf{ OF   ORBIFOLD CHARTS      }}}

\vskip 0.5 cm

  \centerline{Hakeem A. Othman$^1$ and Santanu Acharjee$^2$}
\centerline{$^
1$Department of Mathematics, AL-Qunfudhah University college, Umm Al-Qura University,
KSA.}
\centerline{$^1$Department of Mathematics, Rada'a College of Education and Science, Albaydha University, Albaydha, Yemen.}
 \centerline{$^2$Department of Mathematics, Gauhati University, Guwahati-781014, Assam, India.}

 \centerline{e-mails: $^1$haoali@uqu.edu.sa, hakim$_{-}$albdoie@yahoo.com}
 
  \centerline{$^2$sacharjee326@gmail.com}



\begin{abstract}
\footnotesize{  \noindent In this paper we introduce the notion of
Hurewicz fibrations in the class of embedding maps of  orbifold
charts by giving the concept of E-fibration embedding. We study
the fundamental properties of this concept such as the
restriction, product and its relationship with Hurewicz
fibration, etc. Furthermore, we introduce the notion of lifting
functions of  E-fibration embedding and  study  preserving projection property of these lifting functions. 
\quad\\
 \emph{ 2020 AMS classifications}: 55P05, 57R18, 57N35, 14E25, 32C05.}

\end{abstract}


\section{Introduction}
The concept of orbifold has been described using various
mathematical constructions and contexts. It was first introduced
by Satake \cite{Satake1, Satake2}. He called them
V-manifolds. He defined orbifolds as topological spaces with an
atlas of charts. He viewed orbifolds as a generalization of
manifolds. Orbifolds and manifolds   are described by charts. An
orbifold chart under any space $X$ is defined as a triple $(X_{U},
G_{U}, \Gamma_{U})$, where $X_{U}$ is open set in $R^{n}$, $G_{U}$
is a finite group of homeomorphisms of $X_{U}$ and $\Gamma_{U}$ is
a map. One of the issues with the atlas definition is that there
is no canonical notion of map between orbifolds. Satake introduced
maps of orbifolds \cite{Satake1, Satake2} and called them embedding maps which are considered as
generalizations of smooth maps of manifolds.  One may refer to \cite{BrodskyScepin, Fantechi,  VafaWitten, AminHakeem} for more about orbifold.\\

\noindent In this paper, Section 2 introduces the concept of
E-fibration embedding and  studies  the fundamental properties of
E-fibration embedding such as the restriction property, product
property, relationship between E-fibration embedding maps and
Hurewicz fibrations. In Section 3, we introduce the notion of
lifting functions of E-fibration embedding by giving the concept
of E-lifting function and regular E-lifting function. In Section
4, we show preserving projection property for E-lifting
functions. Throughout this paper all spaces   will be assumed to be
Hausdorff spaces. For any space $X$,  $X^{I}$ denotes the set of
all continuous functions since space (paths) from  $I=[0,1]$ into $X$. We take
$X^{I}$ with the  compact-open topology. For all $x\in X$, by
$\widetilde{x}$ we mean the constant path at a point $x$. For two
paths $\alpha, \beta\in X^{I}$ with $\alpha(1)=\beta(0)$, by
$\alpha\star\beta$ we mean the path in $X$ defined by
\[(\alpha\star\beta)(t)=\left\{%
\begin{array}{ll}
    \alpha(2t), & 0\leq t \leq \frac{1}{2}; \\
    \beta(2t-1), & \frac{1}{2}\leq t \leq 1. \\
\end{array}%
\right.    \]
\begin{defn}\label{Pr1}
   \emph{\cite{ Hatcher, Spanier}  For two continuous functions $f,g: X\to Y $, the function $f$ is called a \emph{
 homotopic}  to $g$ and
write $f\simeq g$ if there exists a  continuous function
 $H: X\times I\to Y$ satisfies  $H(x,0)=f(x)$ and $H (x,1)=g(x)$. }
\end{defn}

\begin{defn}\label{Pr1+1}
   \emph{\cite{Hurewicz1}   A map $f:X\to Y$ is called a
\emph{Hurewicz fibrations}  if for every space $Z$ and two  maps
$h:Z\to X$ and  $H:Z\times I\to Y$ with $H_{0}=f\circ h$, there
exists a map $F:Z\times I\to X$ such that $F_{0}=h$ and $f\circ
F=H$. }
\end{defn}


\noindent By a \emph{topological group} $G$ we mean     a group
$G$ together with a topology on $G$ such that the functions
$(g,g')\to gg'$ and  $g\to g^{-1}$ are continuous of a product
space $G\times G$ into a space $G$ and   of a space  $G$ into $G$,
respectively.  The \emph{action} of   $G$ on any space $X$ is
defined as a map $G\times X\to X$ denoted by $(g,x)\to gx$ such
that $g(g'x)=(gg)x$ and $1x=x$ for all $g,g'\in G$ and
$x\in X$.\\

\noindent  For action $G\times X\to X$ of a topological group $G$
on a  space $X$ and for $x\in X$, we mean by the \emph{orbit set}
of $x$ is the set $G(x)=\{gx\in X: g\in G\}$ and the \emph{orbit
space} $X/G$ is the set of all orbits $G(x)$ in $X$ endowed with
the quotient topology with respect to the \emph{natural orbit map}
$X\to X/G$.

\begin{defn}\label{Pr2}
\emph{\cite{AdemLeidaRuan} An \emph{orbifold chart} on topological
space $X$ is a triple $(X_{U}, G_{U}, \Gamma_{U})$, where $X_{U}$
is an open set in a  space $R^{n}$, $G_{U}$ is a finite group of
homeomorphisms of $X_{U}$ and $\Gamma_{U}:X_{U}\to X$ is a map
defined by $\Gamma_{U}=\overline{\Gamma_{U}}\circ p$, where
$p:X_{U}\to X_{U}/G_{U}$ is the orbit map and
$\overline{\Gamma_{U}}:X_{U}/G_{U}\to X$ is a map that induces a
homeomorphism of $X_{U}$ onto an open subset $U$ of $X$. }
\end{defn}
\begin{defn}\label{Or1}
\emph{\cite{AdemLeidaRuan} For a  topological space $X$, an
\emph{embedding}  $f: X_{U} \to X_{U'}$    is a smooth injective
function  from orbifold chart $(X_{U}, G_{U}, \Gamma_{U})$ into
orbifold chart $(X_{U'}, G_{U'}, \Gamma_{U'})$ and it yields a
homeomorphism between $X_{U}$ and $f(X_{U})$ such that
$\Gamma_{U'}\circ f=\Gamma_{U}$. }
\end{defn}

\section{E-fibration embedding maps}

In this section we introduce \emph{E-fibration} embedding map and study related properties of it.
\begin{defn}\label{E1}
\emph{An  embedding   $f: X_{U} \to  X_{U'}$   of  orbifold chart
$(X_{U}, G_{U}, \Gamma_{U})$ into orbifold chart $(X_{U'}, G_{U'},
\Gamma_{U'})$  is called an \emph{E-fibration} if for every space
$Z$ and map  $h:Z\to X_{U}$ and $H:Z\to X^{I}$ with
$H_{0}=\Gamma_{U}\circ h$, there exists a  map $F:Z\to
(X_{U'})^{I}$ such that $F_{0}=f\circ h$ and $\Gamma_{U'}\circ
F_{t}= H_{t}$ for all $t\in I$.}
\end{defn}

\noindent By an \emph{E-triple } $(\Theta_{2}|X_{2},
\Theta_{1}|X_{1}, X_{\Theta})$ we mean three spaces $X_{1}$,
$X_{2}$ and $X$ with three maps $\Theta_{1}:X_{1}\to X$,
$\Theta_{2}:X_{2}\to X$ and $\Theta :X_{2}\to X_{1}$ such that
$\Theta_{2}\circ \Theta=\Theta_{1}$. We say that an  E-triple
$(\Theta_{2}|X_{2}, \Theta_{1}|X_{1}, X_{\Theta})$ has an
\emph{E-fibration property} if for every space $Z$ and maps
$h:Z\to X_{2}$ and  $H:Z\to X^{I}$ with $H_{0}=\Theta_{2}\circ h$,
there exists a  map $F:Z\to (X_{1})^{I}$ such that
$F_{0}=\Theta\circ h$ and $\Theta_{1}\circ
F_{t}= H_{t}$ for all $t\in I$. \\
\begin{thm}\label{E2}
\emph{An  embedding   $f: X_{U} \to  X_{U'}$   of  orbifold chart
$(X_{U}, G_{U}, \Gamma_{U})$ into orbifold chart $(X_{U'}, G_{U'},
\Gamma_{U'})$  is   an  E-fibration embedding  if and only if an
E-triple $(\Gamma_{U}|X_{U}, \Gamma_{U'}|X_{U'}, X_{f})$ has an
E-fibration property.}
\end{thm}
\begin{exa}\label{E3}
\emph{For an  embedding   $f: X_{U} \to  X_{U'}$   of  orbifold
chart $(X_{U}, G_{U}, \Gamma_{U})$ into  $(X_{U'}, G_{U'},
\Gamma_{U'})$, the  E-triple $(\pi_{2}|X_{U}\times X,
\pi_{2}|X_{U'}\times X, X_{f\times id_{X}})$ has an  E-fibration
property, where  $ \pi_{2}$ is the usual  second projection  and
$id_{X}$ is the identity map on $X$.  Note that If $Z$ is any
space, $h:Z\to X_{U}\times X$ is any map, and $H: Z\to X^{I}$ is a
map with $H_{0}=\pi_{2}\circ h$, define  the desired a  map $F:Z
\to (X_{U'}\times X)^{I}$ by \[F_{t}=(f\circ \pi_{1}\circ h)\times H_{t}\]
for all  $t\in I$, where  $ \pi_{1}$ is the usual  first
projection.}
\end{exa}

\noindent In the following theorem, we show  the composition
property  of E-fibrations  and Hurewicz fibrations.
\begin{thm}\label{E4}
\emph{ Let $f: X_{U} \to  X_{U'}$  be  an  E-fibration embedding
and $f': X \to  X'$  be a Hurewicz fibration of $X$ into a space
$X'$.  Then, the
 E-triple $(f'\circ \Gamma_{U}|X_{U}, f'\circ \Gamma_{U'}|X_{U'}, X'_{f})$ has an  E-fibration property.}
\end{thm}
\begin{proof}
Let $Z$ be any space and let  $h:Z\to X_{U}$   and $H:Z\to X^{'I}$
be any two maps  with $H_{0}=(f'\circ \Gamma_{U})\circ h=f'\circ
(\Gamma_{U}\circ h)$. Since $f'$ is  a Hurewicz fibration, then
there exists  a map $F':Z\to X^{I}$ such that
$F'_{0}=\Gamma_{U}\circ h$ and $f'\circ F'_{t}=H_{t}$ for all
$t\in I$. Since $f$ is an E-fibration, then there exists a map
$F:Z\to (X_{U'})^{I}$ such that $F_{0}= f\circ h $ and
$\Gamma_{U'}\circ F_{t}=F'_{t}$ for all $t\in I$. Then, \[(f'\circ
\Gamma_{U'})\circ F_{t}=f'\circ(\Gamma_{U'}\circ F_{t})=f'\circ
F'_{t}=H_{t}\] for all $t\in I$. Hence the
 E-triple $(f'\circ \Gamma_{U}|X_{U}, f'\circ \Gamma_{U'}|X_{U'}, X'_{f})$ has an  E-fibration property.
\end{proof}
\noindent Theorems \ref{E5}, \ref{E6} and \ref{E7} give the
relations between E-fibrations and Hurewicz fibrations.
\begin{thm}\label{E5}
 \emph{ Let $f: X_{U} \to  X_{U'}$  be  an  E-embedding from  orbifold chart
$(X_{U}, G_{U}, \Gamma_{U})$ into   $(X_{U'}, G_{U'},
\Gamma_{U'})$.  If $\Gamma_{U}$ or $\Gamma_{U'}$ is a Hurewicz
fibration then $f$ is  an  E-fibration.}
\end{thm}
\begin{proof}
Let $Z$ be any space and let  $h:Z\to X_{U}$   and $H:Z\to X^{I}$
be any two maps  with $H_{0}=  \Gamma_{U} \circ h$.  If $
\Gamma_{U}$ is  a Hurewicz fibration, then there exists a  map
$F':Z\to (X_{U})^{I}$ such that $F'_{0}= h$ and $\Gamma_{U}\circ
F'_{t}=H_{t}$ for all $t\in I$. Define a  map $F:Z\to
(X_{U'})^{I}$ by $F_{t}=f\circ F'_{t}$ for all $t\in I$. Note that
$F_{0}=f\circ F'_{0}=f\circ h$ and
\[\Gamma_{U'}\circ F_{t}=\Gamma_{U'}\circ(f\circ F'_{t})=(\Gamma_{U'}\circ f)\circ F'_{t}= \Gamma_{U}\circ F'_{t}=H_{t}\] for all $t\in I$.
 Hence,  $ f$ is an  E-fibration.  If    $
\Gamma_{U'}$ is  a Hurewicz fibration and since \[H_{0}=
\Gamma_{U} \circ h= (\Gamma_{U'}\circ f) \circ h=
 \Gamma_{U'}\circ (f  \circ h),\] then there exists a  map $F :Z\to
(X_{U'})^{I}$ such that $F_{0}= f  \circ h$ and $\Gamma_{U'}\circ
F_{t}=H_{t}$ for all $t\in I$.
 Hence,  $ f$ is a  E-fibration.
\end{proof}

\begin{thm}\label{E6}
 \emph{ Let $f: X_{U} \to  X_{U'}$  be  an  E-fibration embedding from  orbifold chart
$(X_{U}, G_{U}, \Gamma_{U})$ into   $(X_{U'}, G_{U'},
\Gamma_{U'})$.  If $f$ is a  surjective, then  $\Gamma_{U'}$  is a
Hurewicz fibration. }
\end{thm}
\begin{proof}
Let $Z$ be any space and let  $h:Z\to X_{U'}$   and $H:Z\to X^{I}$
be any two maps  with $H_{0}=  \Gamma_{U'} \circ h$.  Since $f$ is
a surjective, then there exists a map $g: X_{U'} \to  X_{U}$ such
that $f\circ g=id_{X_{U'}}$.  Hence, \[H_{0}=  \Gamma_{U'} \circ h=
(\Gamma_{U}\circ g) \circ h=  \Gamma_{U}\circ (g \circ h).\] Then,
there exists a  map $F:Z\to (X_{U'})^{I}$ such that
\[F_{0}=f\circ (g \circ h) = (f\circ g) \circ h=  id_{X_{U'}}\circ h= h \] and $\Gamma_{U'}\circ F_{t}= H_{t}$
for all $t\in I$. Thus, $\Gamma_{U'}$  is a Hurewicz fibration.
\end{proof}
\begin{thm}\label{E7}
 \emph{ Let $f: X_{U} \to  X_{U'}$  be  an  E-fibration embedding from  orbifold chart
$(X_{U}, G_{U}, \Gamma_{U})$ into   $(X_{U'}, G_{U'},
\Gamma_{U'})$.  If $f$ is a  homeomorphism, then  $\Gamma_{U}$ and
$\Gamma_{U'}$  are  Hurewicz fibrations. }
\end{thm}
\begin{proof}
Due to theorem 2.6,  $\Gamma_{U'}$  is   a Hurewicz fibration. Let
$Z$ be any space and let  $h:Z\to X_{U}$ and $H:Z\to X^{I}$ be any
two maps  with $H_{0}=  \Gamma_{U} \circ h$.   Then, there exists a
 map $F':Z\to (X_{U'})^{I}$ such that $F'_{0}=f\circ h$ and
$\Gamma_{U'}\circ F'_{t}= H_{t}$ for all $t\in I$. Define a
 map   $F:Z\to (X_{U})^{I}$ by $ F_{t}=f^{-1}\circ F'_{t}$ for
all $t\in I$. It is easy to note that \[F_{0}= f^{-1}\circ F'_{0}= f^{-1}\circ
(f\circ h)=h\] and \[ \Gamma_{U} \circ F_{t}= \Gamma_{U} \circ
(f^{-1}\circ F'_{t})=(\Gamma_{U} \circ  f^{-1})\circ F'_{t}=
 \Gamma_{U'} \circ F'_{t}=H_{t}
\] for
all $t\in I$.  Thus, $\Gamma_{U}$ is a Hurewicz fibration.
\end{proof}
\noindent The following theorem shows  that the product  of two
E-fibrations has an  E-fibration property.
\begin{thm}\label{E8}
\emph{Let $f: X_{U} \to  X_{U'}$ and $g: Y_{V} \to  Y_{V'}$  be  two
 E-fibration embeddings  over two spaces $X$ and $Y$, respectively.    Then, the
 E-triple \[( \Gamma_{U}\times \Gamma_{V}|X_{U}\times Y_{V},  \Gamma_{U'}\times \Gamma_{V'}|X_{U'}\times Y_{V'}, (X\times Y)_{f\times g})\]
 has an  E-fibration property.}
\end{thm}
\begin{proof}
Let $Z$ be any space and let  $h:Z\to X_{U}\times Y_{V}$   and
$H:Z\to (X\times Y)^{I}$ be any two maps  with
$H_{0}=(\Gamma_{U}\times \Gamma_{V})\circ h$. Define two maps
\[H^{X}:Z\to X^{I} \mbox{ and }H^{Y}:Z\to Y^{I}\] by
\[H^{X}_{t}=\pi_{1}\circ H_{t} \mbox{ and }H^{Y}_{t}=\pi_{2}\circ
H_{t}\] for all $t\in I$, where  $\pi_{1}$ and $ \pi_{2}$ are the
usual first and second projections. It is easy to note that
\[H^{X}_{0}=\pi_{1}\circ H_{0}=\pi_{1}\circ[(\Gamma_{U}\times \Gamma_{V})\circ h]=\Gamma_{U}\circ (\pi_{1}\circ h).\] Since $f$ is an  E-fibration, then
there exists a  map $F':Z\to (X_{U'})^{I}$ such that
$F'_{0}=f\circ (\pi_{1}\circ h)$ and $\Gamma_{U'}\circ
F'_{t}=H^{X}_{t}$ for all $t\in I$. Similarly, for an E-fibration
embedding $g$,  there exists a  map $F'':Z\to (Y_{V'})^{I}$ such
that $F''_{0}=g\circ (\pi_{2}\circ h)$ and $\Gamma_{V'}\circ
F''_{t}=H^{Y}_{t}$ for all $t\in I$. Define a  map $F:Z\to
(X_{U'}\times Y_{V'})^{I}$ by $F_{t}=F'_{t}\times F''_{t}$ for all
$t\in I$. We observe that
 \begin{eqnarray*}
F_{0}=F'_{0}\times F''_{0}&=& [f\circ (\pi_{1}\circ h)]\times
[g\circ (\pi_{2}\circ h)]=(f\times g)\circ [(\pi_{1}\times
\pi_{2})\circ h]\\
&=& (f\times g)\circ h.
\end{eqnarray*}
and
\begin{eqnarray*}
(\Gamma_{U'}\times \Gamma_{V'})\circ F_{t}&=& (\Gamma_{U'}\times
\Gamma_{V'})\circ (F'_{t}\times F''_{t})=
 (\Gamma_{U'}\circ F'_{t}) \times ( \Gamma_{V'})\circ
 F''_{t})\\
 &=&  H^{X}_{t} \times H^{Y}_{t}=  (\pi_{1}\circ H_{t})\times
 (\pi_{2}\circ H_{t})\\
 &=&  (\pi_{1}\times
  \pi_{2})\circ H_{t}=  H_{t}
\end{eqnarray*}
 for all $t\in I$. Hence, the proof is completed.
\end{proof}


\begin{thm}\label{E9}
\emph{ Let $f: X_{U} \to  X_{U'}$  be  an  E-fibration embedding
from  orbifold chart $(X_{U}, G_{U}, \Gamma_{U})$ into   $(X_{U'},
G_{U'}, \Gamma_{U'})$ and $S$ be any subspace of $X$. Then, the
E-triple
 \[(\Gamma_{U_{s}}|\Gamma_{U}^{-1}(S),
\Gamma_{U'_{s}}|\Gamma_{U'}^{-1}(S), S_{f_{s}})\] has an
E-fibration property where $\Gamma_{U_{s}}$, $\Gamma_{U'_{s}}$ and
$f_{s}$  are the restriction maps of $\Gamma_{U}$, $\Gamma_{U'}$
and $f$ on $\Gamma_{U}^{-1}(S)$, $\Gamma_{U'}^{-1}(S)$ and
$\Gamma_{U}^{-1}(S)$, respectively.}
\end{thm}
\begin{proof}
Let $Z$ be any space and let  $h:Z\to \Gamma_{U}^{-1}(S)$   and
$H:Z\to S^{I}$ be any two maps  with $H_{0}= \Gamma_{U_{s}}\circ
h$. Let $j_{u}:\Gamma_{U}^{-1}(S)\to X_{U}$,
$j_{u'}:\Gamma_{U'}^{-1}(S)\to X_{U'}$ and $j:S\to X$ be inclusion
maps. Define  a  map  $H':Z\to X^{I}$ by $H'_{t}=j\circ H_{t}$ for
all $t\in I$. Since $f$ is an E-fibration embedding and
\[H'_{0}=j\circ H_{0}=H_{0}=\Gamma_{U_{s}}\circ
h=(\Gamma_{U_{s}}\circ j_{u})\circ h= \Gamma_{U_{s}}\circ ( j_{u}
\circ h), \] then there exists    map $F:Z\to (X_{U'})^{I}$ such
that $F_{0}=f\circ (j_{u}\circ h)$ and $\Gamma_{U'}\circ
F_{t}=H'_{t}$ for all $t\in I$. Since
\[\Gamma_{U'}[F_{t}(z)]=H'_{t}=j\circ H_{t}=H_{t}\in S\] then
$F(z)(t)\in \Gamma_{U'}^{-1}(S)$ for all $z\in Z$, $t\in I$. Hence,
   $F$ is a   homotopy from $Z$ into
$(\Gamma_{U'}^{-1}(S))^{I}$ and it is easy to  note that
 \[F_{0}=f\circ
(j_{u}\circ h)=(f\circ j_{u})\circ h=f_{s}\circ h\] and
\[\Gamma_{U's}\circ F_{t}=(\Gamma_{U'}\circ j_{u'})\circ F_{t}= \Gamma_{U'}\circ (j_{u'} \circ F_{t})=\Gamma_{U'} \circ F_{t}=H_{t}\] for all $t\in I$.
 Then, the E-triple
 $(\Gamma_{U_{s}}|\Gamma_{U}^{-1}(S),
\Gamma_{U'_{s}}|\Gamma_{U'}^{-1}(S), S_{f_{s}})$ has an E-fibration
property.
\end{proof}


\begin{rem}\label{E10}
\emph{For any   embedding   $f: X_{U} \to  X_{U'}$   of  orbifold
chart $(X_{U}, G_{U}, \Gamma_{U})$ into orbifold chart $(X_{U'},
G_{U'}, \Gamma_{U'})$ and for any map $P:X'\to X$ of a space $X'$
into a space $X$,  define the maps $P_{1}:X'(U)\to X'$,
$P_{1}':X'(U')\to X'$ and $P_{f}:X'(U)\to X'(U')$ by
\[P_{1}(x',r)=x',\mbox{ }P_{1}'(x',r')=x' \mbox{ and } P_{f}(x',r)=(x',f(r)),\] respectively, for all $(x',r)\in X'(U)$ and $(x',r')\in X'(U')$, where
\[X'(U)=\{(x',r)\in  X'\times X_{U}: \Gamma_{U}(r)=P(x')\}\] and \[ X'(U')= \{(x',r')\in  X'\times X_{U'}: \Gamma_{U'}(r')=P(x')\}.\]
It is easy to note that for all  $(x',r)\in X'(U)$ we have that \[(P_{1}'\circ
P_{f})(x',r)=x'= P_{1}(x',r).\]  That is, $(P_{1}|X'(U),
P_{1}'|X'(U'), X'_{P_{f}})$  is an E-triple which is called the
\emph{E-pullback} of  an embedding   $f$  by a map $P$.}
\end{rem}

\begin{thm}\label{E11}
\emph{For any   embedding   E-fibration  $f: X_{U} \to X_{U'}$ of
orbifold chart $(X_{U}, G_{U}, \Gamma_{U})$ into orbifold chart
$(X_{U'}, G_{U'}, \Gamma_{U'})$, the E-pullback of $f$  by any map
$P:X'\to X$ has an E-fibration property.}
 \end{thm}
 \begin{proof}

 Let $Z$ be any space. Let
$h':Z\to X'(U)$ and  $H':Z\to X'^{I}$ be two maps  with
$H'_{0}=P_{1}'\circ h'$. Define a map $h:Z\to X_{U}$ by
$h(z)=\pi_{2}(h'(z))$ and a  map $H:Z\to X^{I}$ by $H(z)=P\circ
H'(z)$ for all $z\in Z$. Thus,
\begin{eqnarray*}
  H(z)(0)&=&(P\circ
H'(z))(0) = P(H'(z)(0))=P[(P_{1}'\circ h')(z)]=P(P_{1}'(h'(z)))\\
  &=&  P(\pi_{1}(h'(z)))=\Gamma_{U}(\pi_{2}(h'(z)))=\Gamma_{U}(h(z))
\end{eqnarray*}
for all $z\in Z$. That is, $H_{0}=\Gamma_{U}\circ h$, where
$\pi_{1}$ and $ \pi_{2}$ are the usual first and second
projections. Since $f$ is an E-fibration, then there exists a
 map $F:Z\to X_{U'}^{I}$ such that $F_{0}=f\circ h$ and
$\Gamma_{U'}\circ F_{t}=H_{t}$ for all $t\in I$.
 Define a  map $F':Z\to X'(U')$ by
$F'(z)(t)=[H'(z)(t),F(z)(t)]$ for all $t\in I, z\in Z$. Note
$P'_{1}\circ F'= H'$ and
\begin{eqnarray*}
  F'(z)(0) &=& [H'(z)(0),F(z)(0)]=[ P_{1}'(h'(z)),f(h(z))]=[\pi_{1}(h'(z)),f(\pi_{2}(h'(z)))] \\
   &=&
   P_{f}[\pi_{1}(h'(z)),\pi_{2}(h'(z))]=P_{f}(h'(z))=(P_{f}\circ h')(z)
\end{eqnarray*}
for all $z\in Z$. Thus, $F'_{0}=P_{f}\circ h'$. Hence, the
E-pullback of $f$  by a  map $P:X'\to X$ has an E-fibration
property.
 \end{proof}

\section{E-lifting function }

In this section, we introduce E-lifting function and study some of its properties.

\begin{defn}\label{E1}
Let    $(\Theta_{2}|X_{2}, \Theta_{1}|X_{1}, X_{\Theta})$ be an
E-triple and let
   \[\bigtriangleup \Theta_{1} =\{( \Theta(x_{2}),\alpha)\in X_{1} \times X^I  :
   \Theta_{2}(x_{2})=\alpha(0)\mbox{ for some } x_{2}\in X_{2}\}.\]
   The map $\lambda_{\Theta} : \bigtriangleup \Theta_{1} \to X_{1}^I$ is called
   an E-lifting function of an E-triple $(\Theta_{2}|X_{2}, \Theta_{1}|X_{1}, X_{\Theta})$ if it satisfies the
   following:
   \begin{enumerate}
    \item  $ \lambda_{\Theta}[  \Theta(x_{2}),\alpha](0)= \Theta(x_{2})$   for all
   $( \Theta(x_{2}),\alpha)\in \bigtriangleup \Theta_{1}$;
    \item $[\Theta_{1} \circ \lambda_{\Theta} (  \Theta(x_{2}),\alpha)](t)= \alpha(t)$  for all
   $( \Theta(x_{2}),\alpha)\in \bigtriangleup \Theta_{1}$ and $t\in I$.
   \end{enumerate}
   
   \end{defn}\label{E19}
Now, we discuss the following theorems related to E-lifting function.

\begin{thm}\label{EL1}
   \emph{The E-triple $(\Theta_{2}|X_{2}, \Theta_{1}|X_{1}, X_{\Theta})$ has  E-fibration
property. if and only if it has an E-lifting function.}
\end{thm}
\begin{proof}
 Let the E-triple $(\Theta_{2}|X_{2}, \Theta_{1}|X_{1}, X_{\Theta})$ has  E-lifting function $\lambda_{\Theta}$ and  $Z$ be any space. Let $h:Z\to X_{2}$
 be any map and
   $H:Z\to X^{I}$ be any  map  such that $\Theta_{2}\circ
   h=H_0$. For $z \in Z$, consider a path $\alpha_z :I\to X$ defined by $\alpha_z(t)=H(z,t)$ for all $t\in I$.
   Define a  map $F:Z  \to (X_{1})^{I}$ by
   \[F(z)(t)=\lambda_{\Theta}[( \Theta \circ h)(z), \alpha_z](t)\]
   for all $z\in Z, t\in
   I$.
    Then, we observe that
    \[
  F(z)(0) =\lambda_{\Theta}[( \Theta \circ h)(z), \alpha_z](0) =( \Theta \circ h)(z)\]
 and
    \begin{eqnarray*}
  (\Theta_{1}\circ F)(z)(t) &=& \Theta_{1}\{\lambda_{\Theta}[( \Theta \circ h)(z), \alpha_z](t)\} \\
   &=& \alpha_z(t)=H(z)(t)
\end{eqnarray*}
    for all $z\in Z, t\in
   I$.  Thus,  $(\Theta_{2}|X_{2}, \Theta_{1}|X_{1}, X_{\Theta})$ has  E-fibration
property.

   Conversely, let   $(\Theta_{2}|X_{2}, \Theta_{1}|X_{1}, X_{\Theta})$ has  E-fibration
property.      Let  $h:\bigtriangleup \Theta_{1}\to X_{2}$  be a
map and
   $H:\bigtriangleup \Theta_{1} \to X^{I}$ be a  map defined
   by $h( \Theta(x_{2}),\alpha)=x_{2} $ for all $( \Theta(x_{2}),\alpha)\in \bigtriangleup \Theta_{1}$ and
   $H( \Theta(x_{2}),\alpha)(t)=\alpha(t)$ for all $t\in I$ and $( \Theta(x_{2}),\alpha)\in \bigtriangleup
   \Theta_{1}$.
    Note that
     \begin{eqnarray*}
 H(\Theta(x_{2}),\alpha)(0)&=& \alpha(0) = \Theta_{2}(x_{2})\\
   &=& (\Theta_{2} \circ h)( \Theta(x_{2}),\alpha)
\end{eqnarray*}
  for all  $( \Theta(x_{2}),\alpha)\in \bigtriangleup
   \Theta_{1}$.  Thus,  $
   H_0=\Theta_{2} \circ h$. Since  $(\Theta_{2}|X_{2}, \Theta_{1}|X_{1}, X_{\Theta})$ has  E-fibration
property,
   then there exists a  map $F:\bigtriangleup \Theta_{1}\to X_{1}^{I}$
   such that $\Theta_{1}\circ F_{t}=H_{t}$ and $ \Theta\circ h=F_0$ for all $t\in I$. Now define a
   map $\lambda_{\Theta} : \bigtriangleup \Theta_{1} \to X_{1}^I$ by \[\lambda_{\Theta} [
    \Theta(x_{2}),\alpha](t)= F(  \Theta(x_{2}),\alpha)(t)\] for all $t\in I$ and $( \Theta(x_{2}),\alpha)\in \bigtriangleup
    \Theta_{1}$.
    Note that
 \begin{eqnarray*}
\lambda_{\Theta}[
    \Theta(x_{2}),\alpha](0)&=&F(  \Theta(x_{2}),\alpha)(0) \\
   &=& ( \Theta\circ h)( \Theta(x_{2}),\alpha)=  \Theta(x_{2})
\end{eqnarray*}
for all $( \Theta(x_{2}),\alpha)\in \bigtriangleup
   \Theta_{1}$ and
 \begin{eqnarray*}
\Theta_{1} \circ \lambda_{\Theta} [  \Theta(x_{2}),\alpha](t)&=&(\Theta_{1}\circ F)[  \Theta(x_{2}),\alpha](t) \\
   &=& H ( \Theta(x_{2}),\alpha)(t)=\alpha(t)
   \end{eqnarray*}
   for all $( \Theta(x_{2}),\alpha)\in \bigtriangleup
   \Theta_{1}$ and $t\in I$. Hence, $ \lambda_{\Theta} $ is an E-lifting function of
    $(\Theta_{2}|X_{2}, \Theta_{1}|X_{1}, X_{\Theta})$.
   \end{proof}
\noindent We say that an E-triple  $(\Theta_{2}|X_{2},
\Theta_{1}|X_{1}, X_{\Theta})$ has  an E-\emph{regular fibration
property}
   if it has an \emph{E-regular lifting function} $\lambda_{\Theta} $, i.e. $\lambda_{\Theta} [  \Theta(x_{2}),\Theta_{1}\circ \widetilde{
\Theta(x_{2})} ] =\widetilde{ \Theta(x_{2})}$ for all $
\Theta(x_{2})\in
    X_{1}$.\\


   \noindent We say that an embedding  map  $f: X_{U} \to  X_{U'}$ is  an  E-\emph{ E-regular fibration embedding} if its E-triple
$(\Gamma_{U}|X_{U}, \Gamma_{U'}|X_{U'}, X_{f})$ has  E-regular lifting function.  \\

\begin{exa}\label{EL3}
\emph{In Example \ref{E3}, we can define the  map
$\lambda_{f\times id_{X}} : \bigtriangleup \pi_{2} \to
(X_{U'}\times X)^I$ by
\[\lambda_{f\times id_{X}}[(f\times id_{X})(r,x),\alpha](t)= (f(r),\alpha(t))\]
for all $t\in I, ( (f\times id_{X})(r,x),\alpha)\in \bigtriangleup
\pi_{2}$.  Note that  for $( (f\times id_{X})(r,x),\alpha)\in
\bigtriangleup \pi_{2}$ and $t\in I$, \[
  \lambda_{f\times id_{X}}[ (f\times id_{X})(r,x),\alpha](0) = (f(r),x)= (f\times
  id_{X})(r,x)\] and
\[
  \pi_{2}\circ \lambda_{f\times id_{X}}[ (f\times id_{X})(r,x),\alpha] (t) =\pi_{2}(f(r),\alpha(t)) =
  \alpha(t).\]
     Hence, $\lambda_{f\times id_{X}}$ is an E- lifting
     function of
E-triple $(\pi_{2}|X_{U}\times X, \pi_{2}|X_{U'}\times X,
X_{f\times id_{X}})$. Also, we observe that for $f\times
id_{X}(r,x)\in X_{U'}\times X$,
     \begin{eqnarray*}
  \lambda_{f\times id_{X}}[ (f\times
id_{X})(r,x),\pi_{2}\circ \widetilde{ f\times id_{X}(r,x)}](t) &=&
[f(r), (\pi_{2}\circ \widetilde{ f\times
id_{X}(r,x)})(t)] \\
   &=& (f(r), x) =(f\times id_{X})(r,x)\\&=& \widetilde{ f\times id_{X}(r,x)}(t).
\end{eqnarray*}
Hence, $ \lambda_{f\times id_{X}}$ is an E-regular lifting
function.}
\end{exa}

\section{Preserving projection property }
In this section, we show preserving projection property for E-lifting function.
\begin{thm}\label{PPP1}
\emph{Let   $f: X_{U} \to  X_{U'}$ be any   E-regular fibration
embedding of orbifold chart $(X_{U}, G_{U}, \Gamma_{U})$ into
orbifold chart $(X_{U'}, G_{U'}, \Gamma_{U'})$. Let  $\lambda':[
f(X_{U})]^I\to X_{U}^I$ be a map defined by
$\lambda'(\beta)=\lambda_{f}(\beta(0),\Gamma_{U'}\circ \beta)$ for
all $\beta\in [ f(X_{U})]^I$. If $ f$ is an injective, then
$\lambda'\simeq inclusion: [ f(X_{U})]^I\subset X_{U'}^{I} $
\emph{preserving projection}. This means that there exists a
 map $H:[ f(X_{U})]^I \times I\to X_{U'}^{I}$ between two maps
$\lambda'$ and $inclusion: [ f(X_{U})]^I\subset X_{U'}^{I}$ such
that
\[[\Gamma_{U'}\circ H(\beta,s)](t)= \Gamma_{U'}(\beta(t))\] for all $\beta \in [ f(X_{U})]^I, s,t\in
I$. }
\end{thm}
\begin{proof}
 Since $ f$ is an injective, then for $\beta
\in [ f(X_{U})]^I$ and $s\in I$, there exists exactly one $r$ such
that $ f(r)=\beta(s)$. Hence, $\lambda'$ is well-defined and the
maps in this proof  will be well-defined. For $\beta \in [
f(X_{U})]^I$ and $s\in I$, we can define a path $\beta_{s} \in [
f(X_{U})]^I$ by
\[ \beta_{s}(t)= \left\{
  \begin{array}{c l}
    \beta(t) & \mbox{for} \quad 0\leq t\leq s,\\
    \beta(s)  & \mbox{for}\quad s\leq t \leq 1.
  \end{array}
\right. \] For $\alpha = \Gamma_{U'}\circ \beta$ and $s\in I$, we
can define the path $\alpha^{1-s}\in X^I$ by
\[ \alpha^{1-s}(t)= \left\{
  \begin{array}{c l}
    \alpha(s+t) & \mbox{for} \quad 0\leq t\leq 1-s,\\
    \alpha(1)  & \mbox{for}\quad 1-s\leq t \leq 1.
  \end{array}
\right. \] Define a  map $H:[ f(X_{U})]^I \times I\to
  X_{U'}^{I}$ by
\[  H(\beta,s) (t)= \left\{
  \begin{array}{c l}
    \beta_{s}(t) & \mbox{for} \quad 0\leq t\leq s,\\
    \lambda_{f}(\beta(s),\alpha^{1-s})(t-s)  & \mbox{for}\quad s\leq t \leq 1,
  \end{array}
\right. \] for all $s\in I, \beta\in [ f(X_{U})]^I$. By the
E-regularity of $f$, we observe that
\begin{eqnarray*}
 H(\beta,0)(t) &=& \lambda_{f}(\beta(0),\alpha^{1})(t) =\lambda_{f}(\beta(0),\alpha)(t)\\
   &=&\lambda_{f}(\beta(0),\Gamma_{U'}\circ\beta)(t)= \lambda'(\beta)(t),
\end{eqnarray*}
and $H(\beta,1) (t)=\beta(t) $ for $\beta\in [ f(X_{U})]^I, t\in
I$. Hence, 
\[H(\beta,0)=\lambda'(\beta)\quad\mbox{and}\quad H(\beta ,
1)=\beta\] for all $\beta\in [ f(X_{U})]^I$. Thus,  $\lambda'
\simeq inclusion :[ f(X_{U})]^I\subset X_{U'}^{I} $. Also,
\begin{eqnarray*}
    [\Gamma_{U'}\circ H(\beta , s)](t) &=&\left\{
  \begin{array}{c l}
   \Gamma_{U'}( \beta_{s}(t)) & \mbox{for} \quad 0\leq t\leq s,\\
    \Gamma_{U'}[\lambda_{f}(\beta(s),\alpha^{1-s})(t-s)]  & \mbox{for}\quad s\leq t \leq
    1,
  \end{array}
\right. \\
&=& \left\{
  \begin{array}{c l}
    \Gamma_{U'}( \beta(t)) & \mbox{for} \quad 0\leq t\leq s,\\
    \alpha^{1-s}(t-s)  & \mbox{for}\quad 0\leq t-s \leq 1-s,
  \end{array}
\right.\\
 &=& \left\{
  \begin{array}{c l}
    \Gamma_{U'}( \beta(t)) & \mbox{for} \quad 0\leq t\leq s,\\
    \alpha(s+t-s)  & \mbox{for}\quad 0\leq t-s \leq 1-s,
  \end{array}
\right.\\
 &=& \left\{
  \begin{array}{c l}
    \Gamma_{U'}( \beta(t)) & \mbox{for} \quad 0\leq t\leq s,\\
    \alpha(t)  & \mbox{for}\quad s\leq t \leq 1,
  \end{array}
\right. \\
&=& \left\{
  \begin{array}{c l}
    \Gamma_{U'}( \beta(t)) & \mbox{for} \quad 0\leq t\leq s,\\
    \Gamma_{U'}(\beta(t))  & \mbox{for}\quad s\leq t \leq 1,
  \end{array}
\right. \\
&=& \Gamma_{U'}(\beta(t)).
\end{eqnarray*}
Hence, $\lambda'\simeq inclusion :[ f(X_{U})]^I\subset X_{U'}^{I}$
preserves  projection.
\end{proof}
\begin{cor}\label{PPP2}
\emph{Let   $f: X_{U} \to  X_{U'}$ be any     E-regular fibration
embedding of orbifold chart $(X_{U}, G_{U}, \Gamma_{U})$ into
orbifold chart $(X_{U'}, G_{U'}, \Gamma_{U'})$. Let
$\lambda':X_{U'}^{I}\to X_{U'}^{I}$ be a map defined by
\[\lambda'(\beta)=\lambda_{f}(\beta(0),\Gamma_{U'}\circ \beta)\] for all  $ \beta\in
X_{U'}^{I}$.  If $ f$ is a homeomorphism, then $\lambda'\simeq
id_{X_{U'}^{I}} $ preserves  projection.}
\end{cor}
 \begin{proof}
   Since
 $ f$ is surjective, thus we can put $ [ f(X_{U})]^I = X_{U'}^{I}$. Hence, the proof can be obtained easily.
 \end{proof}


\begin{cor}\label{PPP3}
\emph{In the Theorem \ref{PPP1}, let $ f$ be a homeomorphism and
for a path $\beta\in X_{U'}^{I}$, let $g_{\beta}$ be a path in
$X_{U'}^{I}$ defined by $g_{\beta}(s)= H(\beta,s) (1)$ for all
 $s\in I$. Let $N:X_{U'}^{I}\to X_{U'}^{I}$ be a map
defined by
$N(\beta)=\lambda_{f}(\beta(0),\Gamma_{U'}\circ\beta)\star
g_{\beta} $  for all $\beta \in X_{U'}^{I}$. Then, $N\simeq
id_{X_{U'}^{I}}$ keeping end points fixed.}
\end{cor}
\begin{proof}
For a path $\beta$ in $X_{U'}^{I}$ and $s\in I$, we can define a
path $\beta_{s}$ in $X_{U'}^{I}$ by
 \[\beta_{s}(t)=\beta(s+(1-s)t) \] for all $ t\in I$.
Define the   map $G:X_{U'}^{I}\times I\to
 X_{U'}^{I}$ by
\[G(\beta , s)=H(\beta , s) \circ (g_{\beta})_{s}\]for all $ s\in I,\beta\in X_{U'}^{I}$. Hence, we observe that
\begin{eqnarray*}
 G(\beta ,0) &=& H(\beta ,0) \circ (g_{\beta})_{0} =\lambda_{f}(\beta(0),\Gamma_{U'} \circ \beta) \circ g_{\beta} \\
   &=& N(\beta)
\end{eqnarray*}
and
\[
 G(\beta ,1) = H(\beta ,1) \circ [g_{\beta}]_{1} =\beta \circ \widetilde{\beta(1)}  \simeq
 id_{X_{U'}^{I}}\]
for all $\beta\in X_{U'}^{I}$. Hence, $N\simeq  id_{X_{U'}^{I}}$.
Also we observe that
\begin{eqnarray*}
  G(\beta , s) (0) &=&  [H(\beta , s)\circ(g_{\beta})_{s}] (0) \\
   &=& H(\beta , s)(0)= \beta_{s}(0)= \beta(0),
\end{eqnarray*}
and
\[
  G(\beta ,s) (1) = [H(\beta ,0)\circ(g_{\beta})_{s}](1) =(g_{\beta})_{s}(1)=
  \beta(1).\] for all $\beta\in X_{U'}^{I}$ and $s\in I$.
Hence, $N\simeq  id_{X_{U'}^{I}}$ keeping end points fixed.
\end{proof}

\begin{cor}\label{PPP4}
\emph{Let   $f: X_{U} \to  X_{U'}$ be an    E-regular fibration
embedding of orbifold chart $(X_{U}, G_{U}, \Gamma_{U})$ into
orbifold chart $(X_{U'}, G_{U'}, \Gamma_{U'})$. If   $ f$ is  a
homeomorphism, then  $X_{U'}^{I}$ and $\bigtriangleup \Gamma_{U'}$
are of the same  map type.}
\end{cor}
 \begin{proof}
  We show that the E-lifting function $\lambda_{f}:
 \bigtriangleup \Gamma_{U'} \to X_{U'}^{I} $ is a
   map equivalence.
  We can define the  map $K:X_{U'}^{I}\to \bigtriangleup \Gamma_{U'}$
 by
$K(\beta)=(\beta(0),\Gamma_{U'}\circ \beta) $ for all $\beta\in
X_{U'}^{I}$.  Then, we observe that from Corollary \ref{PPP3}, $
\lambda' = \lambda_{f}\circ K \simeq  id_{X_{U'}^{I}}$. Also, we
observe that  for $( f(r),\alpha)\in \bigtriangleup \Gamma_{U'}$,
\begin{eqnarray*}
  (K\circ \lambda_{f})( f(r),\alpha) &=& K(\lambda_{f}( f(r),\alpha)) \\
   &=& (\lambda_{f}( f(r),\alpha)(0), \Gamma_{U'}[\lambda_{f}( f(r),\alpha)]) \\
   &=&( f(r),\alpha)= id_{\bigtriangleup \Gamma_{U'}}( f(r),\alpha).
\end{eqnarray*}
Hence, $  K\circ \lambda_{f} = id_{\bigtriangleup \Gamma_{U'}}.$
Then, we get that $\lambda_{f}:
 \bigtriangleup \Gamma_{U'}\to X_{U'}^{I} $ is a
   map equivalence.
  \end{proof}
\begin{thm}\label{PPP5}
\emph{Let   $f: X_{U} \to  X_{U'}$ be an    E-regular fibration
embedding  of orbifold chart $(X_{U}, G_{U}, \Gamma_{U})$ into
orbifold chart $(X_{U'}, G_{U'}, \Gamma_{U'})$. Let
  $\lambda', \widehat{ f}:X_{U}^{I}\to X_{U'}^{I}$ be
 two
 maps  defined by
\[\lambda'(\beta)=\lambda_{f}[ f (\beta(0)),\Gamma_{U}\circ \beta] \mbox{ and } \widehat{ f}(\beta)= f \circ \beta  \] for all $ \beta\in
X_{U}^{I}$.
 Then, $\lambda'\simeq  \widehat{ f}$
preserves projection. }
\end{thm}
\begin{proof}
  For $\beta \in X_{U}^{I}$ and $s\in I$, we can define
a path $\beta_{s} \in X_{U}^{I}$ by
\[ \beta_{s}(t)= \left\{
  \begin{array}{c l}
    \beta(t) & \mbox{for} \quad 0\leq t\leq s,\\
    \beta(s)  & \mbox{for}\quad s\leq t \leq 1,
  \end{array}
\right. \] For $\alpha = \Gamma_{U}\circ \beta$ and $s\in I$, we
can define the path $\alpha^{1-s}\in X$ by
\[ \alpha^{1-s}(t)= \left\{
  \begin{array}{c l}
    \alpha(s+t) & \mbox{for} \quad 0\leq t\leq 1-s,\\
    \alpha(1)  & \mbox{for}\quad 1-s\leq t \leq 1,
  \end{array}
\right. \] Define a  map $H:X_{U}^{I} \times I\to
  X_{U'}^{I}$ by
\[  H(\beta,s) (t)= \left\{
  \begin{array}{c l}
     f[\beta_{s}(t)] & \mbox{for} \quad 0\leq t\leq s,\\
    \lambda_{f}[ f(\beta(s)),\alpha^{1-s}](t-s)  & \mbox{for}\quad s\leq t \leq 1,
  \end{array}
\right. \] for all $s\in I, \beta\in X_{U}^{I}$. By the
E-regularity of $f$, we get that
\begin{eqnarray*}
 H(\beta,0)(t) &=& \lambda_{f}[ f(\beta(0)),\alpha^{1}](t) =\lambda_{f}[ f(\beta(0)),\alpha](t)\\
   &=&\lambda_{f}[ f(\beta(0)),\Gamma_{U}\circ \beta](t)= \lambda'(\beta)(t),
\end{eqnarray*}
and $H(\beta,1) (t)=( f\circ \beta)(t)$  for all $\beta\in
X_{U}^{I}, t\in I$. Hence,
\[H(\beta,0)=\lambda'(\beta)\quad\mbox{and}\quad H(\beta ,
1)= f\circ \beta=\widehat{ f}(\beta),\] for all $\beta\in
X_{U}^{I}$, that is, $\lambda' \simeq \widehat{ f} $. Also
\begin{eqnarray*}
   [\Gamma_{U'}\circ H(\beta , s)](t) &=&\left\{
  \begin{array}{c l}
   \Gamma_{U'}\{  f[\beta_{s}(t)]\} & \mbox{for} \quad 0\leq t\leq s,\\
    \Gamma_{U'}\{\lambda_{f}[ f(\beta(s)),\alpha^{1-s}](t-s)\}  & \mbox{for}\quad s\leq t \leq
    1,
  \end{array}
\right. \\
&=& \left\{
  \begin{array}{c l}
   \Gamma_{U'}\{  f[\beta(t)]\} & \mbox{for} \quad 0\leq t\leq s,\\
    \alpha^{1-s}(t-s)  & \mbox{for}\quad 0\leq t-s \leq 1-s,
  \end{array}
\right. \\
&=& \left\{
  \begin{array}{c l}
    \Gamma_{U'}\{  f[\beta(t)]\} & \mbox{for} \quad 0\leq t\leq s,\\
    \alpha(s+t-s)  & \mbox{for}\quad 0\leq t-s \leq 1-s,
  \end{array}
\right. \\
&=& \left\{
  \begin{array}{c l}
    \Gamma_{U'}\{  f[\beta(t)]\} & \mbox{for} \quad 0\leq t\leq s,\\
    \alpha(t)  & \mbox{for}\quad s\leq t \leq 1,
  \end{array}
\right. \\
&=& \left\{
  \begin{array}{c l}
    \Gamma_{U'}\{  f[\beta(t)]\} & \mbox{for} \quad 0\leq t\leq s,\\
    (\Gamma_{U}\circ \beta)(t)\}  & \mbox{for}\quad s\leq t \leq 1,
  \end{array}
\right. \\
&=& \left\{
  \begin{array}{c l}
    \Gamma_{U'}\{  f[\beta(t)]\} & \mbox{for} \quad 0\leq t\leq s,\\
    \Gamma_{U'}\{  f[\beta(t)]\}  & \mbox{for}\quad s\leq t \leq 1,
  \end{array}
\right. \\
&=&\Gamma_{U'}\textbf{}[\widehat{ f}(\beta)(t)]
\end{eqnarray*}
or all $\beta\in X_{U}^{I}, s, t\in I$. Therefore,  $\lambda'\simeq
\widehat{ f}$ preserving projection.
\end{proof}

\section{Conclusion}
In this paper, we introduced E-fibration embedding map in an orbifold chart. Later we established some fundamental properties of E-fibration such as the
restriction, product, etc.  We also studied  the relationship of E-fibration with Hurewicz fibration. Furthermore, we introduced the notion of E-lifting functions and  studied  preserving projection property of these lifting functions. Ideas of homotopy played a crucial role in this paper. Thus, we hope that this paper will open some new dimensions in research related to orbifold and Hurewicz fibration.\\

\section*{Acknowledgement}
The authors would like to thank the Deanship of Scientific Research at Umm Al-Qura University for supporting this work by Grant Code: 22UQU4330052DSR07


\bibliographystyle{plain}

\begin{thebibliography}{10}

\bibitem{AdemLeidaRuan} A. Adem, J. Leida, and Y. Ruan, Orbifolds and Stringy Topology,
Cambridge Tracts in Mathematics, (2007).

\bibitem{BrodskyScepin} N. Brodsky, A. Chigogidze, E.V. Scepin,  Sections of Serre fibrations with 2-manifold fibers, Top.  Appl., 155,  (2008),  773-782.

\bibitem{Fantechi} B. Fantechi, L. Gottsche,  Orbifold cohomology for global quotients,  Duke Math. Jour., 2 (2003), 197-227.

\bibitem{Hurewicz1}
W. Hurewicz, On the concept of fiber space, Proc. Nat. Acad. Sci. USA, 14   (1955), 956-961.

 \bibitem{VafaWitten} C. Vafa and E. Witten, On orbifolds with discrete torsion, Jour. Geom. Phys.,  15 (1995), 189-214.

\bibitem{Hatcher}
A. Hatcher, Algebraic Topology, Cambridge University Press, Cambridge, (2002).

\bibitem{AminHakeem}
A. Saif and H.  A. Othman, Retractions In Homotopy Theory For Finite Topological Semigroups, Tbilisi Math. Jour., 1(14) 2021, pp. 219-232.

\bibitem{Satake1} I. Satake, On a Generalization of the Notion of Manifold,
Proceedings of the National Academy of Science U.S.A., 42 (1956),
359 - 363.

\bibitem{Satake2} I. Satake, The Gauss-Bonnet theorem for
V-manifolds, Journal of the Mathematical Society of Japan, 9
(1957),  464 - 492.

\bibitem{Spanier}
E.H. Spanier, Algebraic Topology, McGraw-Hill, New York, (1966).

\end{thebibliography}

\end{document}